\documentclass[leqno,12pt]{amsart}

\usepackage{amssymb,amsthm,amsmath,amsfonts}

\usepackage{mathtools}
\usepackage{soul}

\usepackage[colorlinks=true, citecolor=RoyalBlue, filecolor=RoyalBlue, linkcolor=RoyalBlue,urlcolor = RoyalBlue, hyperindex]{hyperref}
\usepackage[dvipsnames]{xcolor}

\usepackage[osf,theoremfont]{newpxtext}
\usepackage[type1]{cabin}
\usepackage[bigdelims,vvarbb]{newpxmath}
\usepackage[varqu,varl]{inconsolata}
\usepackage{tikz}
\usepackage{tikz-cd}
\usetikzlibrary{cd}
\usepackage{graphicx}
\usepackage[all,cmtip]{xy}

\usepackage{bm}
\usepackage{pifont}
\usepackage{upgreek}

\usepackage{eucal}
\usepackage{ulem}
\usepackage{stmaryrd}
\usepackage[capitalise]{cleveref}

\theoremstyle{plain}
\newtheorem{Theoremx}{Theorem}

\newtheorem{theorem}{Theorem}[section]
\theoremstyle{definition}

\theoremstyle{definition}

\newtheorem{lemma}[theorem]{Lemma}

\newtheorem{corollary}[theorem]{Corollary}
\newtheorem{claim}[theorem]{Claim}
\newtheorem{question}[theorem]{Question}

\theoremstyle{definition}

\newtheorem{remark}[theorem]{Remark}
\theoremstyle{remark}

\newtheorem{example}[theorem]{Example}

\numberwithin{equation}{theorem}

\crefname{Theoremx}{Theorem}{Theorems}
\crefname{claim}{Claim}{Claims}
\crefname{question}{Question}{Questions}

\usepackage[headings]{fullpage}

\usepackage{setspace}

\usepackage{enumerate}
\usepackage{calc}

\usepackage{bbding}

\usepackage{verbatim}
\usepackage{alltt}

\usepackage[backend=biber, style=alphabetic, sorting=anyt]{biblatex}
\DeclareMathOperator{\Div}{{div}}

\newcommand{\fm}{\mathfrak{m}}
\newcommand{\Q}{\mathbb{Q}}
\DeclareMathOperator{\Diff}{Diff}
\newcommand{\NN}{\mathbb{N}}

\newcommand{\FF}{\mathbb{F}}

\DeclareMathOperator{\depth}{depth}
\DeclareMathOperator{\Spec}{{Spec}}
\DeclareMathOperator{\Tot}{{Tot}}
\DeclareMathOperator{\Hom}{Hom}
\DeclareMathOperator{\Supp}{{Supp}}
\DeclareMathOperator{\height}{{ht}}

\newcommand{\fp}{\mathfrak{p}}

\renewbibmacro{in:}{}

\begin{document}

\makeatletter
\title[On \texorpdfstring{$F$}{F}-pure inversion of adjunction]{On \texorpdfstring{$F$}{F}-pure inversion of adjunction}
\author{Thomas Polstra}
\address{Department of Mathematics, University of Alabama, Tuscaloosa, AL 35487 USA}
\email{tmpolstra@ua.edu}
\thanks{Polstra was supported by NSF Grant DMS \#2101890 and by a grant from the Simons Foundation, Grant Number 814268, MSRI}
\author{Austyn Simpson}
\address{Department of Mathematics, University of Michigan, Ann Arbor, MI 48109, USA}
\email{austyn@umich.edu}
\thanks{Simpson was supported by NSF postdoctoral fellowship DMS \#2202890.}
\author{Kevin Tucker}
\address{Department of Mathematics, University of Illinois at Chicago, Chicago, IL 60607, USA}
\email{kftucker@uic.edu}
\thanks{Tucker was supported in part by NSF Grant DMS \#2200716.}
\maketitle

\begin{abstract}
We analyze adjunction and inversion of adjunction for the $F$-purity of divisor pairs in characteristic $p > 0$.  In this vein, we give a complete answer for principal divisors under $\mathbb{Q}$-Gorenstein assumptions but without divisibility restrictions on the index. We also give a detailed analysis relating the $F$-purity of the pairs $(R,\Delta + D)$ and that of $(R_D, \Diff_D(\Delta))$ motivated by Kawakita's log canonical inversion of adjunction via reduction to prime characteristic.
\end{abstract}

\section{Introduction}\label{sec:intro} Let $(R,\fm,k)$ be a local ring of prime characteristic $p>0$ and $R\to F_*R$ be the Frobenius map, where $F_*R$ denotes the Frobenius pushforward of $R$. For simplicity assume that $R$ is complete or $F$-finite (\emph{i.e.} $F_*R$ is a finite $R$-module). The ring $R$ is said to be \emph{$F$-pure} if $R\to F_*R$ splits. Let $x\in \fm$ be a nonzerodivisor of $R$. Under suitable hypotheses, Kawakita's breakthrough result on the inversion of adjunction of log canonical singularities \cite{Kaw07} (when viewed through the lens of reduction to prime characteristic via \cite{HW02}) predicts that the following are equivalent:
\begin{enumerate}
    \item The map $R\xrightarrow{\cdot F_*x^{p-1}} F_*R$ is split, \emph{i.e.} \emph{the pair $(R,\Div_R(x))$ is $F$-pure};\label{intro:adjunction}
    \item $R/xR$ is $F$-pure.\label{intro:inversion}
\end{enumerate}
The implications $(\ref{intro:adjunction})\Rightarrow(\ref{intro:inversion})$ and $(\ref{intro:inversion})\Rightarrow(\ref{intro:adjunction})$ 
are special cases of adjunction and inversion of adjunction of $F$-purity, respectively. While the forward direction requires no additional hypotheses, note that the converse $(\ref{intro:inversion})\Rightarrow(\ref{intro:adjunction})$ is known to fail if $R$ is not $\Q$-Gorenstein in light of counterexamples by Fedder \cite{Fed83} and Singh \cite{Sin99}. More generally, we may ask for the above equivalence after incorporating an effective $\Q$-divisor $\Delta$. Even after imposing $\Q$-Cartier assumptions on $K_R + \Delta$, however, the presence of $p$-torsion introduces additional subtleties that often require new methods to overcome.

The first contribution of the present article is the following positive solution to $F$-pure inversion of adjunction along principal ideals for log $\Q$-Gorenstein pairs $(R,\Delta)$ provided the denominators of the coefficients of $\Delta$ are prime to $p$ (but without any divisibility restrictions on the index of $K_R + \Delta$).

\begin{Theoremx}[\Cref{mainthm:principal with boundary}]
\label{mainthm:principal}
Let $(R,\fm,k)$ be an excellent local ring of prime characteristic $p>0$ and $x\in\fm$ a nonzerodivisor such that $R/xR$ is $(G_1)$ and satisfies Serre's condition $(S_2)$. Let $\Delta\geq0$ be an effective $\Q$-divisor of $R$ with components disjoint from $\Div_R(x)$ such that $K_R+\Delta$ is $\Q$-Cartier and $(p^e-1)\Delta$ is integral for all $e\gg 0$ and divisible. Then the pair $(R,\Delta+\Div_R(x))$ is $F$-pure if and only if $(R/xR,\Delta|_{\Div(x)})$ is $F$-pure.
\end{Theoremx}

Note that \Cref{mainthm:principal} applied to the case that $R$ is a regular ring and $\Delta=0$ is Fedder's Criterion for a hypersurface to be $F$-pure \cite[Theorem~1.12]{Fed83}. Moreover, \cref{mainthm:principal} is a significant improvement over previous work by the first two named authors \cite[Theorem A]{PS22} which affirmatively settled the weaker question of whether $F$-purity deforms in $\Q$-Gorenstein rings. When $\Delta = 0$, \cref{mainthm:principal} may be derived from \cite{PS22} using a trick commonly attributed to Manivel \cite{Man93} together with an understanding of the behavior of $F$-purity under separable finite covers; for completeness, we give a detailed proof of this in \Cref{sec:manivel}.
However, the cyclic cover techniques employed in \cite{PS22} do not seem to lend themselves to the incorporation of a boundary in a straightforward manner. Our proof of \cref{mainthm:principal} both circumvents these difficulties and recovers the main theorem of \cite{PS22} while showing a more general result.

More generally, in \cref{sec:adjunction,sec:inversion} we turn our attention to $F$-pure adjunction and inversion of adjunction
beyond the case of principal ideals, \textit{i.e.} along a divisor $D$ with components disjoint from those of the boundary divisor $\Delta$. Specifically, we aim to relate the $F$-purity of $(R,\Delta +D)$ with that of $(R_D,\Diff_D(\Delta))$ where $\Diff_D(\Delta)$ is Shokurov's \emph{different}. Our analysis yields the following results.

\begin{Theoremx}[Adjunction of $F$-purity -- \cref{thm:Adjunction of purity}]
    \label{mainthm:adjunction}
Let $(R,\fm,k)$ be an excellent $(S_2)$ and $(G_1)$ local ring of prime characteristic $p>0$ and let $K_R$ be a choice of canonical divisor of $\Spec(R)$. Suppose that $D\geq 0$ is an effective integral $(S_2)$ and $(G_1)$ divisor, and let $\Delta\geq 0$ be an effective $\Q$-divisor on $\Spec R$ whose components are disjoint from $D$ and such that $(p^e-1)\Delta$ is integral for all sufficiently divisible $e\gg 0$. Suppose that $\Delta+K_R+D$ is $\Q$-Cartier. If $(R,\Delta+D)$ is $F$-pure then $(R_D,\Diff_D(\Delta))$ is $F$-pure.
\end{Theoremx}

\begin{Theoremx}[Inversion of Adjunction of $F$-purity along a $\Q$-Cartier divisor -- \cref{cor:inversion along Q-Cartier divisor}]\label{mainthm:inversion}
With the notation and assumptions of \cref{mainthm:adjunction}, suppose further that
\begin{enumerate}[(I)]
    \item $D$ is $\Q$-Cartier;\label{intro:q-cartier-1}
    \item for each $\Q$-Cartier divisor $E$ and $\fp\in D\subseteq \Spec(R)$, we have the inequality 
    \[
    \depth(R(E)_\fp)\geq \min\{\height(\fp), 3\}.
    \]\label{intro:q-cartier-2}
\end{enumerate}
If the pair $(R_D,\Diff_D(\Delta))$ is $F$-pure then the pair $(R,\Delta+D)$ is $F$-pure.

\end{Theoremx}

\noindent
We remark that assumption (\ref{intro:q-cartier-2}) is necessary in \cref{mainthm:inversion}  whenever $D$ is $\Q$-Cartier by \cref{lem:R has to be S3 sort of}. More generally, we provide a characterization for when inversion of adjunction of an $F$-pure pair is satisfied in the absence of properties (\ref{intro:q-cartier-1}) and (\ref{intro:q-cartier-2}). By \cref{cor:RD F-pure and torsion divisors of index p^e} and \cref{lem:How to prove inversion of adjunction}, if $(R_D,\Diff_D(\Delta))$ is $F$-pure then $(R,\Delta+D)$ is $F$-pure if and only if for each $p$-power torsion divisor $E$ of $R$, the $R_D$-module $R(E)/R(E-D)$ is $(S_2)$.

A particularly notable feature of our analysis above is the lack of index restrictions, as prior work by Schwede has shown \cref{mainthm:adjunction,mainthm:inversion} under the assumptions that $\Delta+D+K_R$ is $\mathbb{Q}$-Cartier of index not divisible by $p$ and $D$ is Cartier in codimension $2$ \cite{Sch09}. In particular, \cref{mainthm:principal} was solved by Schwede under the hypothesis that $R$ is $\Q$-Gorenstein of index not divisible by $p$. Note also that the strongly $F$-regular version of inversion of adjunction was settled by Das in \cite{Das15}. We have also strived throughout to avoid unnecessary $F$-finite restrictions.

Those familiar with Kawakita's theorem on log canonical inversion of adjunction may be surprised by the assumption that $D$ is $(S_2)$ and $(G_1)$, which is slightly weaker than assuming $D$ is normal and which is not a necessity in characteristic $0$. Indeed, the content of \cite{Kaw07} is that in equal characteristic $0$ a pair $(R,\Delta + D)$ is log canonical if and only if $(R_{D^N}, \Diff_{D^N}(\Delta))$ is log canonical where $D^N$ denotes the normalization of $D$. However, as with assumption (\ref{intro:q-cartier-2}) in the case that $D$ is $\Q$-Cartier, our assumption is a necessity in prime characteristic.  There are simple counterexamples to \cref{mainthm:inversion} if we consider the $F$-singularities of $D^N$ instead. Indeed, let $R=\FF_2[x,y,z]$ and $D=\Div_R(x^2-y^2z)$. Then the normalization of $R_D=\FF_2[x,y,z]/(x^2-y^2z)$ is regular with $\Diff_{D^N}(0)$ a smooth divisor so that $(R_{D^N}, \Diff_{D^N}(0))$ is $F$-pure, but the pair $(\FF_2[x,y,z], \Div_R(x^2-y^2z))$ is not $F$-pure by Fedder's Criterion (see also \cref{mainthm:principal}). See \cite[Example 8.4]{Sch09} for additional details, as well as the work of Miller and Schwede \cite{MS12} for an analysis of the behavior of $F$-purity via the normalization map.

\section{The \texorpdfstring{$\mathbb{Q}$}{Q}-Gorenstein case}
\label{sec:principal}

Consider a local ring $(R,\fm,k)$. Let $E_R(k)$ be an injective hull of $k$, and suppose that $R\to M$ is a map of $R$-modules. According to \cite[Lemma 2.1(e)]{HH95}, $R\to M$ is pure if and only if the induced map $E_R(k)\to E_R(k)\otimes_R M$ is injective, a fact that we will use repeatedly. The main goal of this section is to give a proof of \cref{mainthm:principal}. However, we first present a simple proof of $F$-pure (principal) inversion of adjunction when $R$ is a Gorenstein ring since it will guide our subsequent investigations.

\begin{example}[Inversion of Adjunction of $F$-purity, the Gorenstein case]
\label{ex:Inversion of Adjunction F-purity Gorenstein rings}
Let $(R,\fm,k)$ be a local $d$-dimensional Gorenstein ring of prime characteristic $p>0$, and $x\in R$ a nonzerodivisor. Let $\Delta\geq 0$ be an effective $\Q$-Cartier $\Q$-divisor with components disjoint from $\Div_R(x)$ such that $(p^e-1)\Delta$ is integral for all sufficiently divisible $e\gg 0$. We will show the following are equivalent:
\begin{itemize}
    \item $(R,\Delta+\Div_R(x))$ is $F$-pure;
    \item $(R/xR,\Delta|_{\Div_R(x)})$ is $F$-pure.
\end{itemize}
Recall that the pair $(R,\Delta+\Div_R(x))$ is $F$-pure if and only if $R\to F^e_*R((p^e-1)(\Delta+\Div_R(x)))$ is pure. There is a commutative diagram
\[
\begin{xymatrix}
    {
    R \ar[r]\ar[d]^{=} & F^e_*R((p^e-1)(\Delta+\Div_R(x)))\ar[d]^{\cdot F^e_*x^{p^e-1}}_{\cong}\\
    R \ar[r]^{\hspace{-3em}\cdot F^e_*x^{p^e-1}} & F^e_*R((p^e-1)\Delta).
    }
\end{xymatrix}
\]
Similarly, $(R/xR,\Delta|_{\Div_R(x)})$ is $F$-pure if and only if $R/xR\to (F^e_*(R/xR))((p^e-1)\Delta|_{\Div_R(x)})$ is pure. To summarize, we aim to show the following are equivalent:
\begin{itemize}
    \item $R\xrightarrow{\cdot F^e_*x^{p^e-1}} F^e_*R((p^e-1)\Delta)$ is pure;
    \item $R/xR\to F^e_*R/xR((p^e-1)\Delta|_{\Div_R(x)})$ is pure.
\end{itemize}

For each $e>0$ let $\Delta_e=(p^e-1)\Delta$. There exists $e_0$ so that for sufficiently divisible $e\gg 0$, $\Delta_e$ is a $p^{e_0}$-torsion integral divisor. Tensoring the map $R\to F^{e_0}_*R$ with $R(\Delta_e)$ and reflexifying over $R$ gives the map
\[
R(\Delta_e)\to F^{e_0}_*R(p^{e_0}\Delta_e)\cong F^{e_0}_*R.
\]
By \cite{Fed83}, $R$ is necessarily $F$-pure, and the $F$-purity of $R$ then gives that the above map $R(\Delta_e)\to F^{e_0}_*R$ is pure. In particular, it follows that $R(\Delta_e)$ is a Cohen-Macaulay $R$-module. Even further, $R(\Delta_e)/xR(\Delta_e)$ is a Cohen-Macaulay $R/xR$-module; as $R/xR$ is $(S_2)$ and $(G_1)$, we may therefore conclude that the $R/xR$-reflexification map is an isomorphism:
 \[
 R(\Delta_e)/xR(\Delta_e)\xrightarrow{\cong }R/xR(\Delta_e|_{\Div_R(x)})=R/xR((p^e-1)\Delta|_{\Div_R(x)}).
 \]

Consider now the following commutative diagram:
\[
\begin{xymatrix}
{
0\ar[r] & F^e_*R(\Delta_e)\ar[r]^{\cdot F^e_*x} & F^e_*R(\Delta_e)\ar[r] & F^e_*R/xR(\Delta_e|_{\Div_R(x)})\ar[r] & 0 \\
0\ar[r] & R\ar[r]^{\cdot x}\ar[u]_{\cdot F^e_*x^{p^e-1}} & R \ar[r]\ar[u]_{F^e}& R/xR \ar[r]\ar[u]_{F^e}& 0.
}
\end{xymatrix}
\]
 Recall that we wish to show the leftmost vertical map is pure if and only if the rightmost vertical map is pure. The modules $R$ and $F^e_*R(\Delta_e)$ are Cohen-Macaulay. Therefore all of the lower local cohomology modules of $R$ and $F^e_*R(\Delta_e)$ vanish and we have the commutative diagram of local cohomology modules whose horizontal arrows are injective:
\begin{equation}
\begin{xymatrix}
{
F^e_*H^{d-1}_\fm(R/xR(\Delta_e|_{\Div_R(x)}))\ar[r]^{\hspace{2em}\subseteq} & F^e_*H^d_\fm(R(\Delta_e)) \\
H^{d-1}_\fm(R/xR)\ar[r]^{\subseteq}\ar[u]^{ F^e} & H^d_\fm(R)\ar[u]_{\cdot F^e_*x^{p^e-1}}.
}\label{eq:Gor-inversion-1}
\end{xymatrix}
\end{equation}
The local cohomology modules $H^{d-1}_\fm(R/xR)\cong(0:_{H^d_\fm(R)}x)\cong E_{R/xR}(k)$ and $H^d_\fm(R)\cong E_R(k)$ are essential extensions of $k\cong (0:_{H^d_\fm(R)}\fm)$. Therefore the left vertical map is injective if and only if the right vertical map is injective. We observe the following isomorphisms of the vertical maps in (\ref{eq:Gor-inversion-1}):
\begin{align*}
\Big(H^{d-1}_\fm(R/xR)\stackrel{F^e}{\rightarrow}& F^e_* H^{d-1}(R/xR(\Delta_e|_{\Div_R(x)}))\Big)\\
&\cong \left(R/xR\rightarrow F^e_*(R/xR(\Delta_e|_{\Div_R(x)}))\right)\otimes_{R/xR} H^{d-1}_\fm(R/xR)
\end{align*}
and
\begin{align*}
    \left(H^d_\fm(R)\stackrel{\cdot F^e_*x^{p^e-1}}{\rightarrow} F^e_* H^d_\fm(R(\Delta_e))\right)&\cong \left(R\xrightarrow{\cdot F^e_*x^{p^e-1}} F^e_*R(\Delta_e)\right)\otimes_R H^d_\fm(R).
\end{align*}
Therefore $R\xrightarrow{\cdot F^e_*x^{p^e-1}}F^e_*R(\Delta_e)$ is pure if and only if $R/xR\to F^e_*(R/xR(\Delta_e|_{\Div_R(x)}))$ is pure, as desired.
\end{example} 

We briefly review the definition that we will use in this article of a $\Q$-Gorenstein ring and of a generalized divisor, as outside of the normal setting these notions may be somewhat unfamiliar. Let $(R,\fm,k)$ be an excellent equidimensional local ring satisfying conditions $(S_2)$ and $(G_1)$. We utilize the language of (generalized) divisors as introduced by Hartshorne in \cite{Har94}. In particular, in this article when we speak of a divisor $\Delta$ on $R$, we shall mean an \emph{almost Cartier divisor} in the language of \emph{loc. cit.} Explicitly, this means that $\Delta$ is represented by a finitely generated $R$-submodule $I\subseteq \Tot(R)$ of the total quotient ring of $R$ (that is, a fractional ideal) such that:
\begin{enumerate}
    \item $I_\fp=\Tot(R)_\fp$ for all minimal primes $\fp\subseteq R$;
    \item $I$ is reflexive, \emph{i.e.} $I\rightarrow \Hom_R(\Hom_R(I,R),R)$ is an isomorphism.
\end{enumerate}

A generalized divisor corresponding to a fractional ideal $I$ is \emph{effective} if $I\subseteq R$. Note that if $R$ is normal, this notion of divisors corresponds to the usual one. We refer the reader to \cite[\S~2]{Har94} or \cite[Appendix A]{MP21} for more details.

We assume that $R$ has a canonical divisor $K_R$, \emph{i.e.} $K_R$ is the class of divisor so that the corresponding fractional ideal $R(K_R)$ is a canonical module of $R$. We say that $R$ is \emph{$\Q$-Gorenstein} if there exists an integer $n>0$ so that $nK_R\sim 0$. The smallest positive integer $n$ with this property is the \emph{index} of $R$. If $R$ is $\Q$-Gorenstein and $x\in R$ is a nonzerodivisor so that $R/xR$ is $(S_2)$ and $(G_1)$ then $R/xR$ is also $\Q$-Gorenstein by \cite[Proposition~2.6]{PS22}. If $x\in R$ is such an element and if $M$ is an $R$-module, then we write $M_{V(x)}$ to denote the $R$-module $M/xM$. If $E$ is a divisor of $R$ then $R(E)$ is the corresponding fractional ideal. If $E$ has components disjoint from $V(x)$ then $E|_{V(x)}$ denotes the class of the restricted divisor along $V(x)$. Let $K_R$ be a choice of canonical divisor of $\Spec(R)$. By prime avoidance we may assume that $K_R$ has components disjoint from $V(x)$. Then the restricted divisor $K_{R/xR}:=(K_R)|_{V(x)}$ is a choice of canonical divisor of $R/xR$. If $E$ is a divisor of $R$ with components disjoint from $V(x)$ then $\Hom_{R_{V(x)}}(\Hom_{R_{V(x)}}(R(E)_{V(x)}, R_{V(x)}),R_{V(x)})\cong R_{V(x)}(E|_{V(x)})$. Notice that if $M$ is an $R$-module and $x$ is a nonzerodivisor on $M$ then there is a short exact sequence
\[
0\to M\xrightarrow{\cdot x} M\to M_{V(x)}\to 0.
\]
In particular, there is a right exact sequence of local cohomology modules
\[
H^{d-1}_\fm(M_{V(x)})\to H^d_\fm(M)\xrightarrow{\cdot x}H^d_\fm(M)\to 0.
\]
\cref{lem:Injective map of local cohomology} is an instance where the above right exact sequence enjoys the property of being short exact. We first record an elementary observation.

\begin{lemma}
\label{lem:Direct summand p^e torsion}
Let $(R,\fm,k)$ be an excellent local $(S_2)$ and $(G_1)$ $F$-pure ring of prime characteristic $p>0$. Suppose that $E$ is a divisor on $R$. Then $R(E)\to F^e_*R(p^eE)$ is a pure map. Moreover, if $E$ is torsion of index $p^e$ and if $x_1,\ldots,x_t$ is a regular sequence on $R$ then $x_1,\ldots, x_t$ is a regular sequence on $R(E)$.
\end{lemma}

\begin{proof}
Consider the pure map $R\to F^e_*R$. We claim that the composition of maps $R(E)\to F^e_*R\otimes_R R(E)\to F^e_*R(p^eE)$ is a pure map. The assumptions that $R$ is $F$-pure, $(G_1)$ and $(S_2)$ are unchanged by completion (by \cite[Corollary 2.3]{MP21}, \cite[Theorem 18.3]{Mat89}, and \cite[\href{https://stacks.math.columbia.edu/tag/0339}{Tag 0339}]{stacks-project} respectively, the latter using the excellence assumption on $R$), as are the assumptions on the divisor $E$. Hence we may assume that $R$ is complete, so there is a splitting of $R\to F^e_*R$. Applying $-\otimes_R R(E)$ to this splitting and reflexifying we find that there is a splitting of $R(E)\to F^e_*R(p^eE)$. In particular, $R(E)\to F^e_*R(p^eE)$ is pure.

Now suppose that $E$ is torsion of index $p^e$ and $x_1,\ldots, x_t$ is a regular sequence on $R$. Then $R(E)\to F^e_*R(p^eE)\cong F^e_*R$. The sequence $x_1^{p^e},\ldots,x_t^{p^e}$ is a regular sequence on $R$ and therefore $x_1,\ldots,x_t$ is a regular sequence on $F^e_*R$. It follows that $x_1,\ldots, x_t$ is a regular sequence on $R(E)$ as well since we have an injection $R(E)\otimes_R R/(x_1,\ldots,x_i)\hookrightarrow F^e_*R\otimes_R R/(x_1,\ldots,x_i)$ for every $i$.
\end{proof}

\begin{lemma}
\label{lem:local cohomology and reflexification}
Let $(R,\fm,k)$ be a $d$-dimensional excellent local ring and $\varphi:M\to N$ a map of $R$-modules for which $\varphi$ is an isomorphism in codimension $1$. Then $H^d_\fm(M)\to H^d_\fm(N)$ is an isomorphism.
\end{lemma}
\begin{proof}
There exists a four term exact sequence
\[
0\to K\to M\xrightarrow{\varphi} N \to C\to 0
\]
so that $K$ and $C$ are not supported in codimension $1$. Therefore $H^{i}_\fm(K)=H^i_\fm(C)=0$ for $i=d,d-1$. Split the above exact sequence into two short exact sequences
\[
0\to K\to M\to M/K\to 0
\]
and 
\[
0\to M/K\xrightarrow{\varphi} N\to C\to 0.
\]
Examine the long exact sequences of local cohomology modules to conclude that $H^d_\fm(M)\cong H^d_\fm(M/K)\cong H^d_\fm(N)$.
\end{proof}

We will frequently use finite generation of certain local cohomology modules when proving our main theorems. To that end, we have:
\begin{lemma}\label{lem:S2 has fg H2}
Let $(R,\fm,k)$ be a complete equidimensional local ring of dimension $d$, and let $M$ be a finitely generated $R$-module satisfying Serre's condition $(S_t)$ for some $t<d$. Then $\ell_R(H^t_\fm(M))<\infty$.
\end{lemma}
\begin{proof}
This is \cite[Exercise 29]{MP21} and follows from the proof of \cite[Lemma 4.5]{MP21}.
\end{proof}

We remark that the equidimensionality assumption in \Cref{lem:S2 has fg H2} is harmless, as this result will only be applied to local rings which are $(S_2)$ and are homomorphic images of Gorenstein local rings --- such rings are always equidimensional.

\begin{lemma}
\label{lem:Injective map of local cohomology}
Let $(R,\fm,k)$ be a $d$-dimensional excellent local $(S_2)$ and $(G_1)$ ring of prime characteristic $p>0$. Suppose that $x\in R$ is a nonzerodivisor such that $R/xR$ is $(S_2)$ and $(G_1)$. Let $E$ be a Weil divisor of $R$ with the property that $R(E)_{V(x)}$ is reflexive. Then the natural map of local cohomology modules
\[
H^{d-1}_\fm\left(R(-E+K_R)_{V(x)}\right)\to H^d_\fm(R(-E+K_R))
\]
is injective.
\end{lemma}

\begin{proof}
Without loss of generality we may assume that $E$ has components disjoint from $V(x)$. In particular, the reflexification 
$R(E)_{V(x)}\to R_{V(x)}(E|_{V(x)})$ is an isomorphism. There is a short exact sequence
\begin{align}
\label{eq:SES with high depth}
0\to R(E)\xrightarrow{\cdot x} R(E)\to R_{V(x)}(E|_{V(x)})\to 0.
\end{align}
The top local cohomology module $H^d_\fm(R(K_R))$ serves as the injective hull of the residue field. Therefore we consider a choice of the Matlis duality functor $\Hom_R(-,H^d_\fm(R(K_R)))$. 
By Tensor-Hom adjunction 
\begin{align*}
\Hom_R(R(E),H^d_\fm(R(K_R)))&\cong  R(-E)\otimes_R H^d_\fm(R(K_R)).
\end{align*}
The local cohomology module $H^d_\fm(R(K_R))$ is a cokernel of a \u{C}ech complex, and tensor products preserve cokernels. Therefore
\[
R(-E)\otimes_R H^d_\fm(R(K_R))\cong  H^d_\fm(R(-E)\otimes_R R(K_R)).
\]
The reflexification map $R(-E)\otimes_R R(K_R)\to R(-E+K_R)$ is an isomorphism in codimension $1$. Therefore by \cref{lem:local cohomology and reflexification}
\[
H^d_\fm(R(-E)\otimes_R R(K_R))\cong H^d_\fm(R(-E+K_R)).
\]
Similarly, the Matlis dual of $R_{V(x)}(E|_{V(x)})$ is
\begin{align*}
    \Hom_R(R_{V(x)}(E|_{V(x)}), H^d_\fm(R(K_R)))&=\Hom_R(R_{V(x)}(E|_{V(x)}), 0:_{H^d_\fm(R(K_R))}x)\\
    &\cong  \Hom_R(R_{V(x)}(E|_{V(x)}), H^{d-1}_\fm(R(K_{R/xR})))\\ &\cong H^{d-1}_\fm(R_{V(x)}(-E|_{V(x)}+K_{R/xR})).
\end{align*}
Therefore the Matlis dual of (\ref{eq:SES with high depth}) is the short exact sequence
{
\footnotesize
\begin{align}
\label{eq:SES of local cohomology modules}
0\to H^{d-1}_\fm(R_{V(x)}(-E|_{V(x)}+K_{R/xR}))\to H^d_\fm(R(-E+K_R))\xrightarrow{\cdot x}H^d_\fm(R(-E+K_R))\to 0.
\end{align}
}
The reflexification map $R(-E+K_R)_{V(x)}\to R_{V(x)}(-E|_{V(x)}+K_{R/xR})$ is an isomorphism at the codimension $1$ points of $R/xR$. Therefore
\begin{align}
\label{eq:isomorphism of local cohomology}
   H^{d-1}_\fm(R(-E+K_R)_{V(x)})\cong H^{d-1}_\fm(R_{V(x)}(-E|_{V(x)}+K_{R/xR})). 
\end{align}
We conclude the proof by plugging (\ref{eq:isomorphism of local cohomology}) into (\ref{eq:SES of local cohomology modules}).
\end{proof}

We are now prepared to prove \Cref{mainthm:principal}.

\begin{theorem}
\label{mainthm:principal with boundary}
Let $(R,\fm,k)$ be an excellent local ring of prime characteristic $p>0$ and $x\in\fm$ a nonzerodivisor such that $R/xR$ is $(G_1)$ and $(S_2)$. Let $\Delta\geq0$ be an effective $\Q$-divisor of $R$ with components disjoint from $\Div_R(x)$ such that $K_R+\Delta$ is $\Q$-Cartier and $(p^e-1)\Delta$ is integral for all $e\gg 0$ sufficiently large and divisible. Then the pair $(R,\Delta+\Div_R(x))$ is $F$-pure if and only if $(R/xR,\Delta|_{\Div(x)})$ is $F$-pure.
\end{theorem}

\begin{proof}
As in the proof of \Cref{lem:Direct summand p^e torsion} we may assume that $R$ is complete. If $\dim(R)\leq 2$ then the assumption that $R/xR$ is $(G_1)$ implies that the ring $R$ is Gorenstein since the property of being Gorenstein deforms \cite[\href{https://stacks.math.columbia.edu/tag/0BJJ}{Tag 0BJJ}]{stacks-project}. The theorem follows for rings of dimension at most $2$ by \cref{ex:Inversion of Adjunction F-purity Gorenstein rings}.

We assume that $\dim(R)\geq 3$. By induction on the dimension, we may assume that both $R$ and $R/xR$ are $F$-pure when localized at any non-maximal prime ideal in $V(x)$. Let $\Delta_e=(p^e-1)\Delta$ and consider the commutative diagram with exact rows:
\[
\begin{xymatrix}
{
0\ar[r] & F^e_*R(p^eK_R+\Delta_e) \ar[r]^{\cdot F^e_*x} & F^e_*R(p^eK_R+\Delta_e)\ar[r] & F^e_*R(p^eK_R+\Delta_e)_{V(x)}\ar[r] & 0 \\
0\ar[r] & R(K_R)\ar[r]^{\cdot x}\ar[u]^{\cdot F^e_*x^{p^e-1}} & R(K_R)\ar[r]\ar[u]^{F^e} & R(K_R)_{V(x)}\ar[r]\ar[u]^{F^e} & 0.
}
\end{xymatrix}
\]
If $e\gg 0$ and divisible then $(p^e-1)K_R+\Delta_e$ is an integral $p^{e_0}$-torsion divisor. By \Cref{lem:Injective map of local cohomology} applied to the divisor $E=0$, we obtain the following commutative diagram with exact rows:

{\tiny
\begin{equation}
\begin{xymatrix}
{
 & F^e_*H^{d-1}_\fm(R(p^eK_R+\Delta_e)_{V(x)})\ar[r]^{\xi} & F^e_*H^d_\fm(R(p^eK_R+\Delta_e))\ar[r]^{\cdot F^e_*x} & F^e_*H^d_\fm(R(p^eK_R+\Delta_e))\ar[r] & 0 \\
0\ar[r] & H^{d-1}_\fm(R(K_R)_{V(x)})\ar[r]\ar[u]^{F^e} & H^d_\fm(R(K_R))\ar[r]\ar[u]^{\cdot F^e_*x^{p^e-1}} & H^d_\fm(R(K_R))\ar[r]\ar[u]^{F^e} & 0.
}
\end{xymatrix}\label{eq:thm-A-diagram}
\end{equation}
}
The leftmost vertical map of \Cref{eq:thm-A-diagram} is isomorphic to 
\[
(R_{V(x)}\xrightarrow{F^e}F^e_*(R_{V(x)}(\Delta_e|_{\Div_R(x)})))\otimes_{R/xR} H^{d-1}_\fm(R_{V(x)}(K_{R/xR}))
\]
and the middle vertical map is isomorphic to
\[
(R\xrightarrow{\cdot F^e_*x^{p^e-1}}F^e_*R(\Delta_e))\otimes_R H^d_\fm(R(K_R)).
\]
The local cohomology modules $H^{d-1}_\fm(R(K_R)_{V(x)})$ and $ H^d_\fm(R(K_R))$ are essential extensions of the residue field $k$. If $(R,\Delta+\Div_R(x))$ is $F$-pure then the middle vertical map in (\ref{eq:thm-A-diagram}) is injective, from which it follows that the leftmost vertical map in (\ref{eq:thm-A-diagram}) is injective (\emph{i.e.} that $(R/xR,\Delta|_{\Div(x)})$ is $F$-pure). This concludes the proof of the forward direction.

We now prove the converse (\emph{i.e.} inversion of adjunction), so suppose for the remainder of the proof that $(R/xR,\Delta|_{\Div(x)})$ is $F$-pure. Once we know that the map $\xi$ in (\ref{eq:thm-A-diagram}) is injective, we will be able to conclude that $(R,\Delta+\Div_R(x))$ is $F$-pure by similar reasoning as above. This will follow from the next claim.

\begin{claim}\label{claim:thm-A-claim}
Suppose that $E$ is an integral torsion divisor of index $p^{e_0}$. Then $R(E)_{V(x)}$ is an $(S_2)$ $R/xR$-module.\label{claim:reflexification map is an isomorphism}
\end{claim}

\begin{proof}[Proof of \cref{claim:reflexification map is an isomorphism}]
 By assumption, $R$ is $F$-pure whenever we localize at a non-maximal prime ideal $\fp\in V(x)$. By \cref{lem:Direct summand p^e torsion}, the module $R(E)_{V(x)}$ localizes to an $(S_2)$ $R_{V(x)}$-module for all non-maximal primes $\fp\in V(x)$. Therefore it suffices to show that $R(E)_{V(x)}$ has depth at least $2$. Consider the following commutative diagram:
\[
\begin{xymatrix}
{
\, & \, & F^e_*R\ar[r] & F^e_*R_{V(x)}\ar[r] & 0\\
\, & \, & F^e_*R(p^eE)\ar[r]\ar[u]^{\cong} & F^e_*R(p^eE)_{V(x)}\ar[r]\ar[u]^{\cong}  & 0\\
0\ar[r] & R(E)\ar[r]^{\cdot x} & R(E)\ar[r]\ar[u]^{F^e} & R(E)_{V(x)}\ar[r]\ar[u]^{F^e} & 0.
}
\end{xymatrix}
\]
We cannot yet assert that the rightmost vertical map of the above diagram is split. However, we claim that the rightmost map is split after applying $H^2_\fm(-)$. Indeed, the reflexification map $R(E)_{V(x)}\to R_{V(x)}(E|_{V(x)})$ is an isomorphism on the punctured spectrum of $R/xR$. Therefore the induced map of local cohomology modules
\[
H^2_\fm(R(E)_{V(x)})\to H^2_\fm(R_{V(x)}(E|_{V(x)}))
\]
is an isomorphism. Our assumptions imply that $R/xR$ is $F$-pure. By \cref{lem:Direct summand p^e torsion}, $R_{V(x)}(E|_{V(x)})\to F^e_*R_{V(x)}(p^eE|_{V(x)})\cong F^e_*R_{V(x)}$ is split. Therefore 
\[
H^2_\fm(R(E)_{V(x)})\cong H^2_\fm(R_{V(x)}(E|_{V(x)}))\to F^e_*H^2_\fm(R_{V(x)}(p^eE|_{V(x)}))
\]
is split. Since $R/xR$ is assumed to be $(S_2)$, we know that $\depth(R)\geq 3$, hence $H^2_\fm(R)=0$. Now consider the resulting diagram of local cohomology modules
\begin{equation}
    \begin{xymatrix}
{
\, & \, & 0=F^e_*H^2_\fm(R)\ar[r] & F^e_*H^2_\fm(R_{V(x)})\\
\, & \, & F^e_*H^2_\fm(R(p^eE))\ar[r]\ar[u]^{\cong} & F^e_*H^2_\fm(R_{V(x)}(p^eE|_{V(x)}))\ar[u]^{\cong}\\
H^1_\fm(R(E)_{V(x)})\ar[r]^{\subseteq} & H^2_\fm(R(E))\ar[r]^{\cdot x} & H^2_\fm(R(E))\ar[r]\ar[u]^{F^e} & H^2_\fm(R_{V(x)}(E|_{V(x)}))\ar[u]^{F^e}.
}\label{eq:reflexification-claim-1}
\end{xymatrix}
\end{equation}

The composition of the right vertical maps in (\ref{eq:reflexification-claim-1}) is split and therefore injective. One verifies by chasing the diagram (\ref{eq:reflexification-claim-1}) that $H^2_\fm(R(E))\xrightarrow{\cdot x}H^2_\fm(R(E))$ is an onto map. Since $R(E)$ is an $(S_2)$ $R$-module, $H^2_\fm(R(E))$ is finitely generated by \Cref{lem:S2 has fg H2}. By Nakayama's Lemma we have $H^2_\fm(R(E))=0$, so $H^1_\fm(R(E)_{V(x)})=0$ as needed.
\end{proof}

Combining \Cref{claim:thm-A-claim} with \Cref{lem:Injective map of local cohomology} applied to the divisor $E=-(p^e-1)K_R-\Delta_e$ tells us that $\xi$ is injective. Again, using that $H^{d-1}_\fm(R(K_R)_{V(x)})$ and $ H^d_\fm(R(K_R))$ are essential extensions of $k$, the left vertical map of (\ref{eq:thm-A-diagram}) is injective if and only if the middle vertical map of (\ref{eq:thm-A-diagram}) is injective. Therefore $(R_{V(x)},\Delta|_{\Div_R(x)})$ is $F$-pure if and only if $(R,\Delta+\Div_R(x))$ is $F$-pure.
\end{proof}

\section{Adjunction of \texorpdfstring{$F$}{F}-purity along a divisor}
\label{sec:adjunction}

The primary objective of this section is to prove \cref{mainthm:adjunction} (adjunction of $F$-purity), but we first give a brief description of the different $\Diff_D(\Delta)$.

\subsection{The different}\label{subsec:different}
 Let $R$ be a reduced excellent $(S_2)$ and $(G_1)$ local ring, $D$ an effective integral $(S_2)$ and $(G_1)$ Weil divisor, and $\Delta$ an effective $\Q$-divisor with components disjoint from $D$ such that the divisor class $\Delta+D+K_R$ is $\Q$-Cartier of index $n$. We choose a canonical divisor $K_R$ so that $K_R=-D+G$ for some divisor $G$ with components disjoint from $D$. 

 In particular, if $U \subseteq \Spec(R)$ is a dense open subscheme that is regular, then $(K_R+D)|_U$ is Cartier with components disjoint from $D \cap U$.
 It follows that $(K_R + D)|_{D \cap U}$ is a canonical divisor on $D \cap U$.
 The section $1 \in R(G)$ gives rise to a rational section of $\omega_D$ via the restriction mappings $R(G) \to R(G)|_U \to \omega_{D}|_{D \cap U}$, and we denote by $K_D$ the corresponding canonical divisor. Note that if $D$ is Cartier in codimension $2$, then so too is $K_R + D$ as $D$ is $(G_1)$, and it follows that $K_D = (K_R + D)|_{D}$. 
 
 Suppose $n(K_R + D + \Delta) = \Div_R(f)$ for some $f \in R$, and let $\overline{f}$ be its image in $R_D$. The \emph{different of $\Delta$ along $D$} is the $\Q$-divisor $\Diff_D(\Delta) := \frac{1}{n}\Div_D(\overline{f}) - K_D$ of $R_D$. This is independent of the choices above; see \cite[\S~4.1]{Kol13} for further discussion. We recall some properties enjoyed by $\Diff_D(\Delta)$.
\begin{enumerate}
    \item $\Diff_D(\Delta)$ is effective and $0\sim mn(\Delta+D+K_R)|_D=mn(\Diff_D(\Delta)+K_D)$ for every $m\in\mathbb{Z}$; \label{different-1}
    \item Let $V$ be an irreducible codimension $1$ subvariety of $D$ and view $V$ simultaneously as an irreducible codimension $2$ subvariety of $\Spec(R)$. Then $\Diff_D(\Delta)$ is not supported at $V$ if and only if $R$ and $R_D$ are both regular at $V$ and $V\not\in \Supp(\Delta)$; \label{different-2}
    \item If $D$ is Cartier in codimension $2$ then $\Diff_D(\Delta)= \Delta|_D$.\label{different-3}
\end{enumerate}

\subsection{\texorpdfstring{$F$}{F}-purity} Let $(R,\fm,k)$ be a reduced local ring of prime characteristic $p>0$ with total ring of fractions $K$. Suppose that $R$ is $(S_2)$ and $(G_1)$. If $\Delta$ is a $\Q$-divisor then 
\[
R(\Delta)=\{f\in K\mid \Div_R(f)+\Delta\geq 0\}.
\]
If $\Delta=\sum r_i [R/P_i]$ and $\lfloor \Delta\rfloor =\sum \lfloor r_i\rfloor [R/P_i]$ is the round down divisor, then $R(\Delta)=R(\lfloor \Delta\rfloor)$. If $\Delta$ is effective then the pair $(R,\Delta)$ is \emph{$F$-pure} if for all sufficiently divisible $e\gg 0$ the Frobenius maps $R\to F^e_*R((p^e-1)\Delta)$ are pure. There are several competing notions of $F$-purity as it pertains to pairs --- we direct the reader to \cref{subsec:Closing Remarks} for more details. In particular, we provide a counterexample to sharp $F$-pure inversion of adjunction, at least at the level of generality of \cref{mainthm:inversion} (see \cref{ex:sharp-f-pure}).

\begin{lemma}
\label{lem:Equivalent definitions for an F-pure pair}
Let $(R,\fm,k)$ be an excellent $d$-dimmensional $(S_2)$ and $(G_1)$ local ring of prime characteristic $p>0$, $\Delta$ an effective $\Q$-divisor, and $K_R$ a choice of canonical divisor. Then the following are equivalent:
\begin{enumerate}
    \item for all sufficiently divisible $e\gg 0$ the maps 
    \[
    R\to F^e_*R((p^e-1)\Delta)
    \]
    are pure, i.e the pair $(R,\Delta)$ is $F$-pure.\label{lem:eqv-F-pure-1}
    \item There exists an $e_0$ so that for all sufficiently divisible $e\gg 0$ the maps 
    \[
    R\to F^{e+e_0}_*R(p^{e_0}(p^e-1)\Delta)
    \]
    are pure.\label{lem:eqv-F-pure-2}
    \item for all sufficiently divisible $e\gg 0$ the maps of local cohomology
    \[
    H^d_\fm(R(K_R))\to F^e_*H^d_\fm((p^e-1)\Delta+p^eK_R)\cong (R\to F^e_*R((p^e-1)\Delta))\otimes_R H^d_\fm(R(K_R))
    \]
    are injective.\label{lem:eqv-F-pure-3}
    \item There exists an $e_0$ so that for all sufficiently divisible $e\gg 0$ the maps 
    {\small
    \[
    H^d_\fm(R(K_R))\to F^{e+e_0}_*H^d_\fm(p^{e_0}(p^{e}-1)\Delta+p^{e+e_0}K_R)\cong (R\to F^{e+e_0}_*R(p^{e_0}(p^e-1)\Delta))\otimes_R H^d_\fm(R(K_R))
    \]
    }
    are injective.\label{lem:eqv-F-pure-4}
\end{enumerate}
\end{lemma}

\begin{proof}
For all integers $e,e_0$ we can factor the map
\[
R\to F^{e+e_0}_*R((p^{e+e_0}-1)\Delta)
\]
as
\[
R\to F^e_*R((p^e-1)\Delta)\to F^{e+e_0}_*(p^{e_0}(p^e-1)\Delta)\subseteq F^{e+e_0}_*R((p^{e+e_0}-1)\Delta).
\]
This proves the equivalence of (\ref{lem:eqv-F-pure-1}) and (\ref{lem:eqv-F-pure-2}). 

Cokernels are preserved under tensor product and the top local cohomology of a module is the cokernel of the last non-trivial map of a \v{C}ech complex. Hence the isomorphisms described in (\ref{lem:eqv-F-pure-3}) and (\ref{lem:eqv-F-pure-4}) are valid. Furthermore, as $H^d_\fm(R(K_R))$ serves as the injective hull of the residue field of $R$, purity of the map in (\ref{lem:eqv-F-pure-1}) is equivalent to injectivity of the maps in (\ref{lem:eqv-F-pure-3}). (\ref{lem:eqv-F-pure-2}) $\Leftrightarrow$ (\ref{lem:eqv-F-pure-4}) is similar.
\end{proof}

Let $E$ be a $\Q$-divisor and $D$ an effective integral divisor of an $(S_2)$ and $(G_1)$ ring $R$. We use the following notation:
\[
R(E)_D:=\frac{R(E)}{R(E-D)}.
\]

\begin{lemma}
\label{lem:RD F-pure and torsion divisors of index p^e}
Let $(R,\fm,k)$ be a complete $d$-dimensional $(S_2)$ and $(G_1)$ local ring of prime characteristic $p>0$. Suppose that $D$ is an effective integral divisor such that the pair $(R,D)$ is $F$-pure. If $E$ is an integral Weil divisor of index $p^{e_0}$ then $R(E)_D$ is a direct summand of $F^e_*R_D$ for all sufficiently divisible $e\gg 0$.
\end{lemma}

\begin{proof}
For all sufficiently divisible $e\gg 0$ the composition of the maps
\[
R\to F^e_*R\to F^e_*R((p^e-1)D)
\]
are pure since we are assuming the pair $(R,D)$ is $F$-pure. Let $\varphi_e: F^e_*R((p^e-1)D)\to R$ be a splitting of $R\to F^e_*(R(p^e-1)D)$. Then $\varphi_e(F^e_*R(-D))=R(-D)$. To see this, simply tensor $\varphi_e: F^e_*R((p^e-1)D)\to F^e_*R$ with $R(-D)$ and reflexify. Therefore there are commutative diagrams
\[
\begin{xymatrix}
{
0 \ar[r] & R(-D)\ar[r] & R \\ 
0\ar[r] & F^e_*R(-D)\ar[r]\ar[u] & F^e_*R \ar[u]^{\varphi_e} \\
0 \ar[r] & R(-D)\ar[u]\ar[r] & R\ar[u]
}
\end{xymatrix}
\]
and the composition of the vertical maps are the identity maps on their respective modules. Therefore if we tensor by $R(E)$ and reflexify we find that there is a commutative diagram
\[
\begin{xymatrix}
{
0 \ar[r] & R(E-D)\ar[r] & R(E) \\ 
0\ar[r] & F^e_*R(p^eE-D)\ar[r]\ar[u] & F^e_*R(p^eE) \ar[u]^{\tilde{\varphi}_e} \\
0\ar[r] & F^e_*R(-D)\ar[r]\ar[u]_{\cong} & F^e_*R. \ar[u]_{\cong} 
}
\end{xymatrix}
\]
The maps $\tilde{\varphi}_e$ are splittings of $R(E)\to F^e_*R(p^eE)$ and commutativity of the above diagram provides to us an inclusion
\[
\tilde{\varphi}_e(F^e_*R(p^eE-D))\subseteq R(E-D).
\]
It follows that we can restrict $\varphi_e$ to $F^e_*R(p^eE)_D\cong F^e_*R_D$ and produce maps $\tilde{\varphi}_{e,D}:F^e_*R_D\to R(E)_D$ that are splittings of 
\[
R(E)_D\to F^e_*R_D.
\]
In particular, $R(E)_D$ can be realized as a direct summand of $F^e_*R_D$ as claimed.
\end{proof}

\begin{corollary}
\label{cor:RD F-pure and torsion divisors of index p^e}
Let $(R,\fm,k)$ be an excellent $d$-dimensional $(S_2)$ and $(G_1)$ local ring of prime characteristic $p>0$. Suppose that $D$ is an effective integral divisor such that the pair $(R,D)$ is $F$-pure and $R_D$ is $(S_2)$ and $(G_1)$. If $E$ is an integral Weil divisor of index $p^{e_0}$ then the reflexification map $R(E)_D\to R_D(E|_D)$ is an isomorphism. 
\end{corollary}

\begin{proof}
By \cref{lem:RD F-pure and torsion divisors of index p^e} the module $R(E)_D$ is a direct summand of $F^e_*R_D$. We are assuming that $R_D$ is $(S_2)$ and $(G_1)$. Therefore $F^e_*R_D$ is $(S_2)$ and so is any direct summand, hence the relexification map $R(E)_D\to R_D(E|_D)$ is an isomorphism.
\end{proof}

\begin{lemma}
\label{lem:R has to be S3 sort of}
    Let $(R,\fm,k)$ be an $(S_2)$ and $(G_1)$ $d$-dimensional local ring of prime characteristic $p>0$. Suppose that there exists a nonzero reduced $\Q$-Cartier divisor $D$ such that $(R,D)$ is $F$-pure. If $E$ is a $\Q$-Cartier divisor then
    \[
    \depth(R(E))\geq \min\{3,d\}.
    \]
\end{lemma}
\begin{proof}
    We are assuming that $R$ is $(S_2)$. Therefore we may assume $R$ is of dimension at least $3$ and show that $H^2_\fm(R(E))=0$. The composition of maps
    \[
    R\to F^e_*R\subseteq F^e_*R((p^e-1)D)
    \]
    is pure. Tensor with $R(E)$ and reflexify to conclude the composition of maps
    \[
    R(E)\to F^e_*R(p^eE)\subseteq F^e_*R(p^eE+(p^e-1)D)
    \]
    is pure. Consequently, the composition of maps of local cohomology modules
    \[
    H^2_\fm(R(E))\to H^2_\fm(F^e_*R(p^eE))\subseteq H^2_\fm(F^e_*R(p^eE+(p^e-1)D))
    \]
    is injective. To conclude that $H^2_\fm(R(E))=0$, it suffices to show that the inclusion 
    \[
    R(p^eE)\subseteq R(p^eE+(p^e-1)D)
    \]
    induces the $0$-map of local cohomology modules 
    \[
    H^2_\fm(R(p^eE))\to H^2_\fm(R(p^eE+(p^e-1)D))
    \]
    for all sufficiently divisible $e\gg 0$.

The divisor $E$ is $\Q$-Cartier. Therefore there exists an $e_0$ so that for all sufficiently divisible $e\gg 0$ we have $R(p^eE)\cong R(p^{e_0}E)$. Let $g_e$ be an element of the total ring of fractions of $R$ so that $R(p^eE)\xrightarrow{\cdot g_e}R(p^{e_0}E)$ is an isomorphism. Suppose that $D$ has $\Q$-Cartier index $n$ and $nD=\Div(f)$. For all $e\gg 0$ there exists $q_e\geq 1$ and $0\leq r_e<n$ such that $p^e-1=q_en+r_e$. Observe that $q_e\to \infty$ as $e\to \infty$. There are commutative diagrams:
    \[
    \begin{xymatrix}
        {
        R(p^eE)\ar[d]_{\cdot g_e}^{\cong}\ar[rr]^{\hspace{-3em}\subseteq}& & R(p^eE+(p^e-1)D)\ar[d]_{\cong}^{\cdot g_ef^{q_e}} \\
        R(p^{e_0}E)\ar[r]^{\cdot f^{q_e}} &R(p^{e_0}E)\ar[r]^{\subseteq} 
        & R(p^{e_0}E+r_eD).
        }
    \end{xymatrix}
    \]
    Consider the induced maps of local cohomology modules:
     \[
    \begin{xymatrix}
        {
        H^2_\fm(R(p^eE))\ar[d]^{\cong}_{\cdot g_e}\ar[rr]& &H^2_\fm(R(p^eE+(p^e-1)D))\ar[d]_{\cong}^{\cdot g_ef^{q_e}} \\
        H^2_\fm(R(p^{e_0}E))\ar[r]^{\cdot f^{q_e}} &H^2_\fm(R(p^{e_0}E))\ar[r]& H^2_\fm(R(p^{e_0}E+r_eD)). \\
        }
    \end{xymatrix}
    \]
    The module $R(p^{e_0}E)$ is $(S_2)$, $f\in \fm$, and therefore $H^2_\fm(R(p^{e_0}E))\xrightarrow{\cdot f^{q_e}}H^2_\fm(R(p^{e_0}E))$ is the $0$-map for all $e\gg 0$. Therefore the map $H^2_\fm(R(p^{e}E))\to H^2_\fm(R(p^{e}E+(p^e-1)D))$ is the $0$-map as it can be factored through the $0$-map.
\end{proof}

\subsection{Adjunction of \texorpdfstring{$F$}{F}-purity}

\begin{theorem}[Adjunction of $F$-Purity]
\label{thm:Adjunction of purity}
Let $(R,\fm,k)$ be an excellent $d$-dimensional $(S_2)$ and $(G_1)$ local ring of prime characteristic $p>0$. Suppose that $D$ is an integral $(S_2)$ and $(G_1)$ divisor. Let $\Delta$ be an effective $\Q$-divisor of $R$ with components disjoint from $D$ such that $(p^e-1)\Delta$ is integral for all sufficiently divisible $e\gg 0$. Suppose that $\Delta+D+K_R$ is $\Q$-Cartier. If $(R,\Delta+D)$ is $F$-pure then $(R_D,\Diff_D(\Delta))$ is $F$-pure.
\end{theorem}

\begin{proof}
There exists an integer $e_0$ so that for all sufficiently divisible $e\gg 0$ 
\[
\Delta_e:=(p^e-1)(\Delta+D+K_R)
\]
is an integral divisor of index $p^{e_0}$. By \cref{lem:RD F-pure and torsion divisors of index p^e} $R(-\Delta_e)_D$ is a direct summand of $F^e_*R_D$ for all sufficiently divisible $e\gg 0$. Therefore $R(-\Delta_e)_D$ is an $(S_2)$ $R_D$-module, \emph{i.e.}\footnote{Notice that this isomorphism requires that $D$ is $(G_1)$.}
\[
R(-\Delta_e)_D\cong R_D(-\Delta_e|_D)=R_D(-(p^e-1)(\Diff_D(\Delta)+K_D)).
\]

We have shown the existence of the following short exact sequences for all sufficiently divisible $e\gg 0$:
\begin{align}
    \label{eq:important SES}
    0\to R(-\Delta_e-D)\to R(-\Delta_e)\to R_D(-\Delta_e|_D)\to 0.
\end{align}
The Matlis dual of (\ref{eq:important SES}) produces the short exact sequence
\begin{align}
\label{eq:SES of lc 1}
0\to H^{d-1}_\fm(R_D(\Delta_e|_D+K_D))\to H^d_\fm(R(\Delta_e+K_R))\to H^d_\fm(R(\Delta_e+D+K_R))\to 0,
\end{align}
\emph{i.e.} the kernel of the natural map 
\[
H^d_\fm(R(\Delta_e+K_R))\to H^d_\fm(R(\Delta_e+D+K_R))
\]
induced by the inclusion $R(\Delta_e+K_R)\subseteq R(\Delta_e+D+K_R)$ is $H^{d-1}_\fm(R_D(\Delta_e|_D+K_D)$.
Similarly, the Matlis dual of the short exact sequence
\[
0\to R(-D)\to R\to R_D\to 0
\]
is 
\begin{align}
\label{eq:SES of lc 2}
0\to H^{d-1}_\fm (R_D(K_D))\to H^d_\fm(R(K_R))\to H^d_\fm(R(D+K_R))\to 0.
\end{align}
Therefore the kernel of $H^d_\fm(R(K_R))\to H^d_\fm(R(K_R+D))$ induced from the inclusion $R(K_R)\subseteq R(K_R+D)$ is $H^{d-1}_\fm(R_D(K_D))$.

Consider the following commutative diagrams with exact rows:
\[
\begin{xymatrix}
{
    0\ar[r] & F^{e}_*R(\Delta_e+K_R)\ar[r] & F^{e}_*R(\Delta_e+D+K_R)\ar[r] &  F^{e}_*R(\Delta_e+D+K_R)_D \ar[r] & 0 \\ 
0\ar[r] & R(K_R)\ar[r]\ar[u]_{F^e} & R(D+K_R)\ar[r]\ar[u]_{F^e} & R(D+K_R)_D\ar[r]\ar[u]_{F^e} & 0.
}
\end{xymatrix}
\]
The short exact sequences (\ref{eq:SES of lc 1}) and (\ref{eq:SES of lc 2}) give to us commutative diagrams of local cohomology modules with exact rows
{\footnotesize
\begin{align}
\label{eq:cd of lc}
\begin{xymatrix}
{
    0\ar[r] & F^{e}_*H^{d-1}_\fm(R_D(\Delta_e|_D+K_D))\ar[r] & F^{e}_*H^d_\fm(R(\Delta_e+K_R))\ar[r] &  F^{e}_*H^d_\fm(R(\Delta_e+D+K_R)) \ar[r] & 0 \\ 
0\ar[r] & H^{d-1}_\fm(R_D(K_D))\ar[r]\ar[u]_{F^e} & H^d_\fm(R(K_R))\ar[r]\ar[u]_{F^e} & H^d_\fm(R(D+K_R))\ar[r]\ar[u]_{F^e} & 0.
}
\end{xymatrix}
\end{align}
}
In light of \cref{lem:local cohomology and reflexification}, the left vertical map in (\ref{eq:cd of lc}) is isomorphic to $$(R_D\to F^e_*R((p^e-1)\Diff_D(\Delta)))\otimes_R H^{d-1}_\fm(R_D(K_D))$$ and the right vertical map is isomorphic to $(R\to F^e_*R((p^e-1)(\Delta+D)))\otimes_R H^d_\fm(R(K_R))$. We are assuming that $(R,\Delta+D)$ is $F$-pure, therefore the middle map is injective by \cref{lem:Equivalent definitions for an F-pure pair}(\ref{lem:eqv-F-pure-3}). Since the middle vertical map is injective, so too is the left vertical map. Therefore the pair $(R_D,\Diff_D(\Delta))$ is $F$-pure by \cref{lem:Equivalent definitions for an F-pure pair}(\ref{lem:eqv-F-pure-1}).
\end{proof}

\begin{remark}
The proof of \cref{thm:Adjunction of purity} contains an important observation that will be necessary to proving the theorem's converse. Namely, if $(R,\Delta+D)$ is $F$-pure and $(p^e-1)\Delta$ is integral for all sufficiently divisible $e\gg 0$ then
\[
R(-\Delta_e)|_D\to R_D(-\Delta_e|_D)
\]
is an isomorphism for all such $e$.
\end{remark}

\section{Inversion of Adjunction of \texorpdfstring{$F$}{F}-purity}
\label{sec:inversion}

\begin{lemma}[Key Lemma]
\label{lem:How to prove inversion of adjunction}
Let $(R,\fm,k)$ be an excellent $d$-dimensional $(S_2)$ and $(G_1)$ local ring of prime characteristic $p>0$. Suppose that $D$ is a reduced $(S_2)$ and $(G_1)$ divisor. Let  $\Delta$ be an effective $\Q$-divisor with components disjoint from $D$. Suppose that $\Delta+D+K_R$ is $\Q$-Cartier and $(p^e-1)\Delta$ is integral for all sufficiently divisible $e\gg 0$. Suppose that $(R_D,\Diff_D(\Delta))$ is $F$-pure. Then the following are equivalent:
\begin{enumerate}
    \item The pair $(R,\Delta+D)$ is $F$-pure;\label{key-lemma-condition-1}
    \item If $E$ is a $p$-power torsion divisor then $R(E)_D=R(E)/R(E-D)$ is an $(S_2)$ $R_D$-module.\label{key-lemma-condition-2}
\end{enumerate}
\end{lemma}

\begin{proof}
If the pair $(R,\Delta+D)$ is $F$-pure then the pair $(R,D)$ is $F$-pure. Therefore if $E$ is a torsion divisor of index $p^{e_0}$ then $R(E)_D$ is a direct summand of $F^e_*R_D$ for all sufficiently divisible $e\gg0$ by \cref{lem:RD F-pure and torsion divisors of index p^e}. We are assuming that $R_D$ is $(S_2)$. Therefore $R(E)_D$ is $(S_2)$ since it is a direct summand of an $(S_2)$ $R_D$-module. This completes the forward direction $(\ref{key-lemma-condition-1})\Rightarrow(\ref{key-lemma-condition-2})$.

 We introduce notation that will be used in the proof of the converse $(\ref{key-lemma-condition-2})\Rightarrow(\ref{key-lemma-condition-1})$. For each $e\in\NN$ let $\Delta_e=(p^e-1)(\Delta+D+K_R)$. There exists an integer $e_0$ so that for all sufficiently divisible $e\gg 0$ the divisor $-\Delta_e$ is torsion of index $p^{e_0}$. By assumption, the reflexification map $R(-\Delta_e)_D\to R_D(-\Delta_e|_D)$ is an isomorphism for all sufficiently divisible $e\gg 0$. There are short exact sequences
\[
0\to R(-\Delta_e-D)\to R(-\Delta_e)\to R_D(-\Delta_e|_D)\to 0.
\]
Matlis duality provides to us short exact sequences (again using that $R_D$ is $(S_2)$)
\[
0\to H^{d-1}_\fm(R_D(\Delta_e|_D+K_D))\to H^d_\fm(R(\Delta_e+K_R))\to H^d_\fm(R(\Delta_e+D+K_R))\to 0.
\]
Consequently, the kernel of the natural map $H^d_\fm(R(\Delta_e+K_R))\to H^d_\fm(R(\Delta_e+D+K_R))$ induced by the inclusion $R(\Delta_e+K_R)\to R(\Delta_e+D+K_R)$ is $H^{d-1}_\fm(R_D(\Delta_e|_D+K_D))$. Similarly, there are short exact sequences,
\[
0\to R(-D)\to R\to R_D\to 0
\]
and Matlis duality provides to us short exact sequences
\[
0\to H^{d-1}_\fm(R_D(K_D))\to H^d_\fm(R(K_R))\to H^d_\fm(R(K_R+D))\to 0.
\]
Therefore, the commutative diagram 
\[
\begin{xymatrix}
{
0 \ar[r] & F^e_*R(\Delta_e+K_R)\ar[r] & F^e_*R(\Delta_e+K_R+D) \ar[r] & F^e_*R(\Delta_e+K_R+D)_D\ar[r] & 0 \\
0 \ar[r]& R(K_R)\ar[r]\ar[u] & R(K_R+D)\ar[r]\ar[u] & R(K_R)_D\ar[r]\ar[u] & 0
}
\end{xymatrix}
\]
induces the commutative diagram
\begin{equation}
\tiny
\begin{xymatrix}
{
0 \ar[r] & F^e_*H^{d-1}_\fm(R_D(\Delta_e|_D+K_D)) \ar[r]& F^e_*H^d_\fm(R(\Delta_e+K_R))\ar[r] &  F^e_*H^d_\fm(R(\Delta_e+K_R+D))\ar[r] & 0 \\
0 \ar[r] & H^{d-1}_\fm(R_D(K_D))\ar[r]\ar[u] & H^d_\fm(R(K_R))  \ar[u] \ar[r]& H^d_\fm(R(K_R+D))\ar[r]\ar[u] & 0.
}
\end{xymatrix}\label{eq:inversion-1}
\end{equation}
The left vertical map in (\ref{eq:inversion-1}) is isomorphic to
\begin{equation}
(R_D\to F^e_*R_D((p^e-1)\Diff_D(\Delta)))\otimes_{R_D} H^{d-1}_\fm(R_D(K_D)).\label{eq:inversion-2}
\end{equation}

Since $(R_D,\Diff_D(\Delta))$ is $F$-pure, the map in (\ref{eq:inversion-2}) is injective by \cref{lem:Equivalent definitions for an F-pure pair}. Since $H^{d-1}_\fm(R_D(K_D))$ and $H^d_\fm(R(K_R))$ are essential extensions of the residue field, a simple chase of a socle element in the diagram (\ref{eq:inversion-1}) shows that the middle map is injective. The middle vertical map of (\ref{eq:inversion-1}) isomorphic to
\[
(R\to F^e_*R((p^e-1)(\Delta+D)))\otimes_R H^d_\fm(R(K_R)).
\]
In particular, the pair $(R,\Delta +D)$ is $F$-pure by \cref{lem:Equivalent definitions for an F-pure pair}.
This completes the proof of $(\ref{key-lemma-condition-2})\Rightarrow(\ref{key-lemma-condition-1})$.
\end{proof}

\begin{corollary}[Inversion of adjunction along a $\Q$-Cartier divisor]
\label{cor:inversion along Q-Cartier divisor}
    Let $(R,\fm,k)$ be an excellent $(S_2)$ and $(G_1)$ ring of prime characteristic $p>0$ and let $K_R$ be a choice of canonical divisor of $\Spec(R)$. Suppose that $D\geq 0$ is an effective integral $(S_2)$ and $(G_1)$ divisor, and let $\Delta\geq 0$ be an effective $\Q$-divisor on $\Spec R$ whose components are disjoint from those of $D$ and such that $(p^e-1)\Delta$ is integral for all sufficiently divisible $e\gg 0$. Suppose that $\Delta+K_R+D$ is $\Q$-Cartier. Suppose that 
    \begin{enumerate}[(I)]
    \item $D$ is $\Q$-Cartier;
    \item for each $\Q$-Cartier divisor $E$ and $\fp\in D\subseteq \Spec(R)$ that 
    \[
    \depth(R(E)_\fp)\geq \min\{\height(\fp), 3\}.
    \]
\end{enumerate}
If $(R_D,\Diff_D(\Delta))$ is $F$-pure then $(R,\Delta+D)$ is $F$-pure.
\end{corollary}

\begin{proof}
    By \cref{lem:How to prove inversion of adjunction} we only require that $R(E)_D$ is an $(S_2)$ $R_D$-module whenever $E$ is a $p$-power torsion divisor. We first remark that if $r\in R$ and $s\in R(E)$, then $rs\in R(E-D)$ if and only if $r\in R(-D)$. Therefore $R(E)_D=R(E)/R(E-D)$ is a torsion-free $R_D$-module. In particular, if $\dim R\leq 2$ then $R_D$ is at most $1$-dimensional and $R(E)_D$ is a torsion-free (and therefore a maximal Cohen-Macaulay) $R_D$-module. By induction, we may assume that $\dim R\geq 3$ and aim to show that $H^1_\fm(R(E)_D)=0$.
    
    There are short exact sequences
    \[
    0\to R(E-D)\to R(E)\to R(E)_D\to 0
    \]
    and therefore $H^1_\fm(R(E))\subseteq H^2_\fm(R(E-D))$. The divisor $E-D$ is torsion and therefore the latter local cohomology module is $0$ by assumption.
\end{proof}

\Cref{lem:R has to be S3 sort of} demonstrates a necessary condition for $F$-pure inversion of adjunction, namely
    \[
    \depth(R(E)_\fp)\geq \min\{\height(\fp), 3\}
    \]
for all $\Q$-Cartier divisors $E$ and $\fp\in D\subseteq \Spec(R)$. This assumption is vacuous whenever the ambient ring is strongly $F$-regular by \cite[Corollary~3.3]{PS14}. Moreover, we can replace this hypothesis with the milder assumption that for all $\fp\in D\subseteq \Spec(R)$
    \[
    \depth(R_\fp)\geq \min\{\height(\fp), 3\}
    \] and still obtain $F$-pure inversion of adjunction under the hypothesis that $D$ is a $p$-power torsion divisor.

\begin{corollary}[Inversion of adjunction along a $p$-power torsion divisor]
    Let $(R,\fm,k)$ be an excellent $(S_2)$ and $(G_1)$ ring of prime characteristic $p>0$ and let $K_R$ be a choice of canonical divisor of $\Spec(R)$. Suppose that $D\geq 0$ is an effective integral $(S_2)$ and $(G_1)$ divisor, and let $\Delta\geq 0$ be an effective $\Q$-divisor on $\Spec R$ whose components are disjoint from those of $D$ and such that $(p^e-1)\Delta$ is integral for all sufficiently divisible $e\gg 0$. Suppose that $\Delta+K_R+D$ is $\Q$-Cartier. Suppose further that $D$ is a $p$-power torsion divisor and that for all $\fp\in D\subseteq \Spec(R)$ that
     \[
    \depth(R_\fp)\geq \min\{\height(\fp), 3\}.
    \] 
    If $(R_D,\Diff_D(\Delta))$ is $F$-pure then $(R,\Delta+D)$ is $F$-pure.
\end{corollary}

\begin{proof}
As in the proof of \cref{cor:inversion along Q-Cartier divisor}, we may assume that $R$ is of dimension at least $3$. By induction on the dimension of $R$, we may assume that $(R,\Delta+D)$ is $F$-pure when localized at a non-maximal point of $D$. In light of \cref{lem:How to prove inversion of adjunction}, the reflexification map $R(E)_D\to R_D(E|_D)$ is an isomorphism on the punctured spectrum. Therefore it suffices to show $H^1_\fm(R(E)_D)=0$ to conclude that $R(E)_D$ is an $(S_2)$ $R_D$-module.

 As the reflexification map $R(E)_D\to R_D(E|_D)$ is an isomorphism on the punctured spectrum we observe
\begin{itemize}
    \item $H^2_\fm(R(E)_D)\cong H^2_\fm(R_D(E|_D))$, and
    \item $H^2_\fm(R(E)_D)\to F^e_*H^2_\fm(R_D(p^eE|_D))$ is an injective map as $R_D$ is $F$-pure. 
\end{itemize}
Consider the following commutative diagram:
\[
\begin{xymatrix}
{
0\ar[r] & F^e_*R(p^eE-D) \ar[r] & F^e_*R(p^eE) \ar[r] & F^e_*R(p^eE)_D \ar[r] & 0 \\
0 \ar[r] & R(E-D)\ar[u] \ar[r] & R(E)\ar[u] \ar[r]\ar[u] & R(E)_D\ar[r]\ar[u] & 0
}
\end{xymatrix}
\]
and the resulting commutative diagram of local cohomology modules:
{\footnotesize
\[
\begin{xymatrix}
{
\footnotesize
0 \ar[r] & F^e_*H^1_\fm(R(p^eE)_D) \ar[r] & F^e_*H^2_\fm(R(p^eE-D)) \ar[r] & F^e_*H^2_\fm(R(p^eE)) \ar[r] & F^e_*H^2_\fm(R(p^eE)_D) \\
0 \ar[r] & H^1_\fm(R(E)_D) \ar[r]\ar[u]_\alpha & H^2_\fm(R(E-D)) \ar[r]\ar[u]_\beta & H^2_\fm(R(E)) \ar[r]^\zeta\ar[u]_\gamma & H^2_\fm(R(E)_D).\ar[u]_\delta
}
\end{xymatrix}
\]
}
If $e\gg 0$ then $p^eE\sim 0$ in which case $H^1_\fm(R(p^eE)_D)=H^2_\fm(R(p^eE))=H^2_\fm(R)=0$ by assumption and the diagram simplifies as
\[
\begin{xymatrix}
{
0 \ar[r] & 0 \ar[r] & F^e_*H^2_\fm(R(p^eE-D)) \ar[r] & 0 \ar[r] & F^e_*H^2_\fm(R(p^eE)_D) \\
0 \ar[r] & H^1_\fm(R(E)_D) \ar[r]\ar[u]_\alpha & H^2_\fm(R(E-D)) \ar[r]\ar[u]_\beta & H^2_\fm(R(E)) \ar[r]^{\zeta}\ar[u]_\gamma & H^2_\fm(R(E)_D)\ar[u]_\delta
}
\end{xymatrix}
\]
The map $\delta$ is injective. Therefore $\zeta$ is the $0$-map and we have a surjection $H^2_\fm(R(E-D))\twoheadrightarrow H^2_\fm(R(E))$. We mention that up to this point, we have not utilized that $D$ is a $p$-power torsion divisor.

We repeat the above argument with the $p$-power torsion divisor $E-D$ to obtain a surjection $H^2_\fm(R(E-2D))\twoheadrightarrow H^2_\fm(R(E-D))$. More generally, for any $n\in\NN$ we obtain surjections
\begin{equation*}
    H^2_\fm(R(E-nD))\twoheadrightarrow H^2_\fm(R(E-(n-1)D))\twoheadrightarrow\cdots\twoheadrightarrow H^2_\fm(R(E-D))\twoheadrightarrow H^2_\fm(R(E)).
\end{equation*}
Choose $n$ so that $nD=\Div_R(f)\sim 0$. Then there is a commutative diagram
\[
\begin{xymatrix}
    {
    H^2_\fm(R(E-nD))\ar[r]^{\cong} & H^2_\fm(R(E))\\
    H^2_\fm(R(E))\ar[r]^{\cdot f}\ar[u]^{\cdot f}_{\cong} & H^2_\fm(R(E))\ar[u]^{=}.
    }
\end{xymatrix}
\]

Consequently, $H^2_\fm(R(E))=fH^2_\fm(R(E))$. The module $H^2_\fm(R(E))$ is finitely generated by \Cref{lem:S2 has fg H2} since $R(E)$ is $(S_2)$ and $\dim R\geq 3$. By Nakayama's Lemma the module $H^2_\fm(R(E))=0$. Similarly, $H^2_\fm(R(E-D))=0$ because $E-D$ is a $p$-power torsion divisor. Therefore $H^1_\fm(R(E)_D)=0$ as needed since $H^1_\fm(R(E)_D)$ is a submodule of $H^2_\fm(R(E-D))=0$.
\end{proof}

\subsection{Closing Remarks on \texorpdfstring{\Cref{mainthm:adjunction,mainthm:inversion}}{Theorems B and C}}\label{subsec:Closing Remarks} Techniques used here will be useful in solving other cases of the Inversion of Adjunction of $F$-purity problem. For example, there are cases to consider whenever the boundary divisor $\Delta$ does not have the property that $(p^e-1)\Delta$ is integral for sufficiently divisible $e\gg 0$. Arguments using our techniques will require additional assumptions on the modules $R((p^e-1)\Delta)=R(\lfloor(p^e-1)\Delta\rfloor)$ for sufficiently divisible $e\gg 0$. For example, one could develop a version of \cref{lem:How to prove inversion of adjunction} under the additional hypotheses that the for $e\gg 0$ and divisible that the round-down divisor $\lfloor (p^e-1)(\Delta+D+K_R)\rfloor$ is $\Q$-Cartier.

Alternatively, one can consider novel roundings of the divisors $(p^e-1)\Delta$ when defining a competing notion of an $F$-pure pair $(R,\Delta)$. For example, Schwede defined a pair $(R,\Delta)$ to be \emph{sharply $F$-pure} if for all sufficiently divisible $e\gg 0$ the maps 
\[
R\to F^e_*R(\lceil(p^e-1)\Delta\rceil)
\]
are pure, \cite{Sch08}. Sharp $F$-purity has a distinct advantage whenever $\Delta+D$ is $\Q$-Cartier:
\[
\Hom_R(F^e_*R(\lceil(p^e-1)\Delta\rceil),R)\cong F^e_*R(-(p^e-1)(\Delta+K_R)).
\]
Regardless, there are simple counterexamples to sharp $F$-pure inversion of adjunction, even if the ambient ring $R$ is nonsingular.

\begin{example}\label{ex:sharp-f-pure}
Let $R=\FF_p\llbracket x,y,z\rrbracket$, $\Delta=\frac{1}{p}[R/(x+y^p+z^p)]$, and $D=V(x)$. The pair $(R_D, \Delta|_D)$ is sharply $F$-pure but $(R,\Delta+\Div_R(x))$ is not sharply $F$-pure. To see this first notice that $[R/(x+y^p+z^p)]|_D=p[R_D/(y+z)]$. Therefore $\Delta|_D=[R_D/(y+z)]$. The pair $(R_D, \Delta|_D)$ is (sharply) $F$-pure. Indeed, $R_D$ is isomorphic to the regular local ring $\FF_p\llbracket y,z\rrbracket$ and $(y+z)^{p^e-1}\not\in \fm^{[p^e]}$ for all $e$. 

Observe that
\[
\lceil(p^e-1)\Delta\rceil=\left\lceil \frac{p^e-1}{p}\right\rceil\left[\frac{R}{(x+y^p+z^p)}\right]=p^{e-1}\left[\frac{R}{(x+y^p+z^p)}\right].
\]
In particular, the pair $(R,\Delta+D)$ is not sharply $F$-pure as $(x+y^p+z^p)^{p^{e-1}}x^{p^e-1}\in\fm^{[p^e]}$ for all $e\in\NN$.
\end{example}

We conclude this section by highlighting an application of our results (and those of \cite{PS22}) to a question of Enescu.
\begin{question}[see {\cite[Question 3.5]{Ene00}} and surrounding discussion]\label{question:non-F-pure-section}
Does there exist an excellent local $F$-rational ring $(R,\fm)$ such that for all nonzerodivisors $x\in R$, the ring $R/xR$ is not $F$-pure?
\end{question}
\begin{example}
    By either \cref{mainthm:principal} or \cite[Theorem A]{PS22}, any $\Q$-Gorenstein isolated singularity $(R,\fm)$ which is $F$-rational but not $F$-pure will provide an affirmative answer to \Cref{question:non-F-pure-section}. In particular, the ring
\begin{equation*}
    R=\left(\frac{\FF_2[x,y,z,w]}{(x^3+y^3+z^3+w^3)}\right)^{(n)}_\fm,\hspace{.5cm} n\geq 2
\end{equation*}
(where $(-)^{(n)}$ denotes the $n$th Veronese subring and $\fm$ is its homogeneous maximal ideal) is such an example by \cite[Example 6.3]{Sin00}.
\end{example}

\section{Manivel's Trick}\label{sec:manivel}
Recall the following theorem on deformation of $F$-purity:
\begin{theorem}
[\cite{PS22}]\label{thm:deformation}
    Let $(R,\fm,k)$ be a local $F$-finite $\Q$-Gorenstein ring of prime characteristic $p>0$. Suppose that $x\in \fm$ is a nonzerodivisor such that $R/xR$ is $(G_1)$, $(S_2)$, and $F$-pure. Then $R$ is $F$-pure.
\end{theorem}
We will demonstrate in this section how a special case of \Cref{mainthm:principal} follows from \Cref{thm:deformation} using a variant of a trick due to Manivel  \cite{Man93}.

\begin{proof}[Proof of \cref{mainthm:principal} when $R$ is $F$-finite and $\Delta=0$] 
Assume that $(R,\fm,k)$ is an $F$-finite $\mathbb{Q}$-~Gorenstein local ring of prime characteristic $p$, $x \in \fm$ is a nonzerodivisor such that $R/xR$ is $(G_1)$ and $(S_2)$, and so that $R/xR$ is $F$-pure. Suppose $n \in \mathbb{Z}_{>0}$ is relatively prime to $p$ and consider the module finite extension ring $U_n := R[Y]/(Y^n - x)$. Note first that $U_n$ is reduced, as $U_n$ is generically \'etale over $R$. Also, $U_n$ is a free $R$-module of rank $n$, from which it follows $U_n$ remains $(S_2)$ and $f$ is a nonzerodivisor on $U_n$. Consequently, if $y$ denotes the image of $Y$ in $U_n$, $y$ is also a nonzerodivisor on $U_n$, and moreover $U_n/(y) = R /xR$ is $(S_2)$, $(G_1)$ and $F$-pure. Since $\Hom_{R}(U_n,R) \cong U_n$, \textit{e.g.} by taking the projection onto the $R$-factor corresponding to $y^{n-1}$, we also see that $U_n$ is $(G_1)$ as it is Gorenstein over the Gorenstein locus of $R$. Similarly, $U_n$ remains $\mathbb{Q}$-Gorenstein. By \Cref{thm:deformation} (applied to the localizations at the finitely many maximal ideals of $U_n$), it follows that $U_n$ is $F$-pure.

Letting $n = p+1$, consider the following diagram
\[\begin{tikzcd}[column sep=large]
	R && {F_*R} \\
	\\
	{U_{p+1}} && {F_*U_{p+1}}.
	\arrow["{1 \mapsto y^p}"', from=1-1, to=3-1]
	\arrow["{1 \mapsto F_*x^{p-1}}", from=1-1, to=1-3]
	\arrow["{F_*1 \mapsto F_*y}", from=1-3, to=3-3]
	\arrow["{1 \mapsto F_*1}"', from=3-1, to=3-3]
	\arrow["\alpha", from=3-1, to=3-3]
	\arrow["\beta", from=1-1, to=3-1]
	\arrow["\gamma"', from=1-1, to=1-3]
	\arrow["\delta"', from=1-3, to=3-3]
\end{tikzcd}\]
To see that this commutes, note that  we have $$\delta(\gamma(1)) = \delta(F_*x^{p-1}) = F_*x^{p-1}y = F_*y^{(p+1)(p-1)}y = F_*y^{p^2}$$
and also 
$$
\alpha(\beta(1)) = \alpha(y^p) = y^p\alpha(1) = y^pF_*1 = F_*y^{p^2}.
$$
Since $U_{p+1}$ is $F$-pure, the bottom rightward arrow $\alpha$ is a split inclusion of $U_{p+1}$-modules, hence also of $R$-modules. The left downward arrow $\beta$ is also a split inclusion of $R$-modules. It follows that the composite mapping $\alpha \circ \beta = \delta \circ \gamma$ is a split inclusion of $R$-modules, and so in particular $\gamma$ is as well. By definition, we see that the pair $(R,\Div_R(x))$ is $F$-pure. 

\end{proof}

\begin{remark}
One can also apply Manivel's original trick in the above proof instead. With the same notation as in the first paragraph of the proof, instead consider the commutative diagram
\[\begin{tikzcd}[column sep=large]
	R && {F^e_*R} \\
	\\
	{U_n} && {F^e_*U_n}
	\arrow["{1 \mapsto y^{n-1}}"', from=1-1, to=3-1]
	\arrow["{1 \mapsto F^e_*x^{(1-1/n)(p^e-1)}}", from=1-1, to=1-3]
	\arrow["{F^e_*1 \mapsto F^e_*y^{n-1}}", from=1-3, to=3-3]
	\arrow["{1 \mapsto F^e_*1}"', from=3-1, to=3-3]
	\arrow["\beta", from=1-1, to=3-1]
	\arrow["\alpha", from=3-1, to=3-3]
	\arrow["\gamma"', from=1-1, to=1-3]
	\arrow["\delta"', from=1-3, to=3-3]
\end{tikzcd}\]
where $e \gg 0$ is sufficiently large and divisible so that $n \vert (p^e - 1)$. Since $U_n$ is $F$-pure, the bottom rightward arrow $\alpha$ is a split inclusion of $U_n$-modules, hence also of $R$-modules. The left downward arrow $\beta$ is also a split inclusion of $R$-modules. It follows that the composite mapping $\alpha \circ \beta = \delta \circ \gamma$ is a split inclusion of $R$-modules, and so in particular $\gamma$ is as well. By definition, we see that the pair $(R,(1-1/n)\Div_R(x))$ is $F$-pure. Letting $n \to \infty$ and using \cite[Theorem 4.9]{Her12}, it follows that $(R, \Div_R(x))$ is $F$-pure.
\end{remark}

\section*{Acknowledgments}
We thank Mircea Musta\c{t}\u{a} for suggesting the use of Manivel's trick in \Cref{sec:manivel}.

\printbibliography

\end{document}